\documentclass[12pt]{article}
\usepackage{amsfonts}
\usepackage{amsmath}
\usepackage{mathrsfs}
\usepackage{color}
\usepackage{tikz}
\usepackage{tkz-graph}
\usepackage{diagbox}
\usepackage{hyperref}
\hypersetup{hypertex=true,
	colorlinks=true,
	linkcolor=blue,
	anchorcolor=blue,
	citecolor=blue}
\usetikzlibrary{chains,fit,shapes}
\usetikzlibrary{calc}
\usetikzlibrary{chains,fit,shapes}
\usetikzlibrary{positioning,calc}
\usepackage{array,booktabs}
\usetikzlibrary{arrows}

\usepackage{authblk}
\usepackage{mathrsfs,amscd,amssymb,amsthm,amsmath,bm,graphicx,psfrag,subfigure,url,indentfirst,verbatim}
\usepackage[skipabove=5pt,skipbelow=5pt,leftmargin=3pt,rightmargin=3pt,linewidth=0.8pt]{mdframed}
\newtheorem{thm}{Theorem}

\newtheorem{lem}{Lemma}

\newtheorem{conj}{Conjecture}
\newtheorem{Obs}{Observation}
\newenvironment{wst}
{\setlength{\leftmargini}{1.5\parindent}
	\begin{itemize}
		\setlength{\itemsep}{-1.1mm}}
	{\end{itemize}}

\begin{document}
	\title{\bf An embedding technique in the study of word-representabiliy of graphs}
	\author[a]{Sumin Huang\thanks{Email: sumin2019@sina.com}}
	\author[b]{Sergey Kitaev\thanks{Email: sergey.kitaev@strath.ac.uk}}
	\author[c]{Artem Pyatkin\thanks{Email: artem@math.nsc.ru}}
	\affil[a]{School of Mathematical Sciences, Xiamen
		University, Xiamen 361005, P.R. China}
	\affil[b]{Department of Mathematics and Statistics, University of Strathclyde, 26 Richmond Street, Glasgow G1 1XH, United Kingdom}
	\affil[c]{Sobolev Institute of Mathematics, Koptyug ave, 4, Novosibirsk, 630090, Russia}
	\date{}
	\maketitle

\begin{abstract}
Word-representable graphs, which are the same as semi-transitively orientable graphs, generalize several fundamental classes of graphs. In this paper we propose a novel approach to study word-representability of graphs using a technique of homomorphisms.
As a proof of concept, we apply our method to show word-representability of the simplified graph of overlapping permutations that we introduce in this paper. For another application, we obtain results on word-representability of certain subgraphs of simplified de Bruijn graphs that were introduced recently by Petyuk and studied in the context of word-representability.  
\end{abstract}

\vspace{2mm} \noindent{\bf Keywords}: simplified de Bruijn type graph, word-representability, semi-transitive orientation, homomorphism
\vspace{2mm}

\section{Introduction}
\subsection{Simplified de Bruijn graph and its subgraphs} 
\noindent Let $A(k)=\{0,1,\ldots,k-1\}$ be a $k$-letter alphabet. For any integers $n\geq 2$ and $m\geq 0$, let $A^n_m(k)$ be the following set of words over $A(k)$:
$$A^n_m(k)=\{x_1x_2\ldots x_n:x_i\in A(k) \text{ and } |x_i-x_{i+1}|\geq m\}.$$
In particular, $A^n_0(k)$ is the set of all words over $A(k)$ of length $n$. 

A {\it de Bruijn graph} $B(n,k)$ is a digraph with vertex set $A^n_0(k)$ having an arc from $x_1x_2\ldots x_n$ to $y_1y_2\ldots y_n$ if and only if $x_{i+1}=y_{i}$ for all $i\in \{1,2,\ldots,n-1\}$. De Bruijn graphs are a useful tool in combinatorics on words \cite{Lothaire} and they find applications in several areas outside of mathematics, for example, in bioinformatics \cite{Pevzner}.

The notion of the {\em simplified de Bruijn graph} is introduced in  \cite{Petyuk}, and we believe it to be a very natural and potentially useful tool in graph theory and its various applications. The simplified de Bruijn graph $S(n,k)$ is the simple graph obtained from $B(n,k)$ by removing orientations and loops and replacing multiple edges between a pair of vertices by a single edge.
Denote by  $S_m(n,k)$ the induced subgraphs of $S(n,k)$ with vertex set $A^n_m(k)$. 
See Figure~\ref{graph-ex-1} for examples of just introduced objects.  

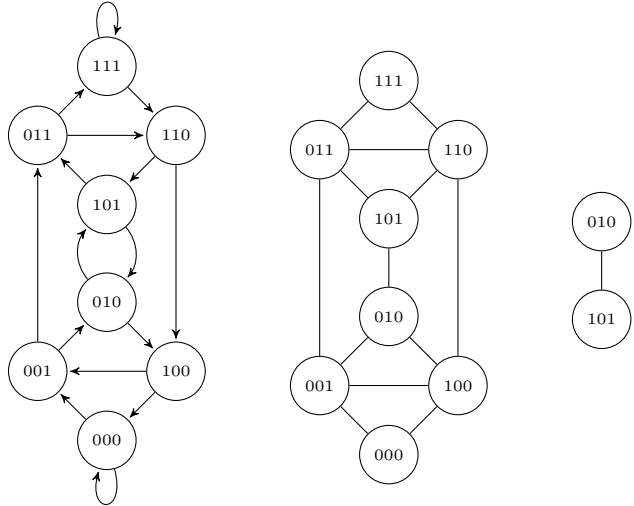
\begin{figure}
\begin{center}
\begin{tabular}{ccccc}


\begin{tikzpicture}[->,>=stealth',shorten >=1pt,node distance=1.3cm,auto,main node/.style={circle,draw,align=center}]

\node[main node] (1) {\tiny{111}};
\node[main node] (2) [below left of=1] {\tiny{011}};
\node[main node] (3) [below right of=1] {\tiny{110}};
\node[main node] (4) [below right of=2] {\tiny{101}};
\node[main node] (5) [below of=4] {\tiny{010}};
\node[main node] (6) [below left of=5] {\tiny{001}};
\node[main node] (7) [below right of=5] {\tiny{100}};
\node[main node] (8) [below left of=7] {\tiny{000}};

\path
(1) edge node {} (3);
\path
(8) edge node {} (6);
\path
(4) edge node {} (2)
      edge [bend left=40] node  {} (5);
\path
(5) edge node {} (7)
     edge [bend left=40] node  {} (4);
\path
(3) edge node  {} (4)
     edge node  {} (7);
\path
(7) edge node  {} (6)
     edge node  {} (8);
\path
(6) edge node  {} (5)
     edge node  {} (2);
\path
(2) edge node  {} (1)
     edge node  {} (3);

\path
(1) edge [loop above] node {} (1);
\path
(8) edge [loop below] node {} (8);

\end{tikzpicture}

& \ \ \ &

\begin{tikzpicture}[node distance=1.3cm,auto,main node/.style={circle,draw,align=center}]

\node[main node] (1) {\tiny{111}};
\node[main node] (2) [below left of=1] {\tiny{011}};
\node[main node] (3) [below right of=1] {\tiny{110}};
\node[main node] (4) [below right of=2] {\tiny{101}};
\node[main node] (5) [below of=4] {\tiny{010}};
\node[main node] (6) [below left of=5] {\tiny{001}};
\node[main node] (7) [below right of=5] {\tiny{100}};
\node[main node] (8) [below left of=7] {\tiny{000}};
\node (9) [below of=8, yshift=0.5cm] {};

\path
(1) edge node {} (3);
\path
(8) edge node {} (6);
\path
(4) edge node {} (2);
\path
(5) edge node {} (7)
     edge node  {} (4);
\path
(3) edge node  {} (4)
     edge node  {} (7);
\path
(7) edge node  {} (6)
     edge node  {} (8);
\path
(6) edge node  {} (5)
     edge node  {} (2);
\path
(2) edge node  {} (1)
     edge node  {} (3);

\end{tikzpicture}

& \ \ \ &

\begin{tikzpicture}[node distance=1.3cm,auto,main node/.style={circle,draw,align=center}]

\node[main node] (1) {\tiny{010}};
\node[main node] (2) [below of=1] {\tiny{101}};
\node (3) [below of=2] {};
\node (4) [below of=3] {};

\path
(1) edge node {} (2);

\end{tikzpicture}

\end{tabular}

\end{center}

\vspace{-0.8cm}
\caption{From left to right: the graphs $B(3,2)$, $S(3,2)$ and $S_1(3,2)$.}\label{graph-ex-1}
\end{figure}

\subsection{Simplified graphs of overlapping permutations}

\noindent The {\em graph of overlapping permutations} $P(n)$ is defined in a way analogous to the de Bruijn graph $B(n,k)$. However, here we require that the head and tail of adjacent permutations have their letters appear in the same relative order. Formally, the vertex set of $P(n)$ is the set of all $n!$ permutations of $\{1,2,\ldots,n\}$, and there is an arc from a permutation $x_1x_2\ldots x_n$ to a permutation $y_1y_2\ldots y_n$ if and only if, for each $2\leq i<j\leq n$, 
either both $x_i<x_j$ and $y_{i-1}<y_{j-1}$ hold or both $x_i>x_j$ and $y_{i-1}>y_{j-1}$ hold. The graph $P(n)$ is instrumental in solving various problems related to permutations, for example, in constructing a universal cycle for permutations \cite{CDG1992}. 

The {\em simplified graph of overlapping permutations} $SP(n)$ is obtained from $P(n)$ by removing orientations and loops and replacing multiple edges between a pair of vertices by a single edge. To our best knowledge, the notion of the simplified graph of overlapping permutations is introduced in this paper for the first time. See Figure~\ref{graph-ex-2} for examples of just introduced objects. 

\begin{figure}
\begin{center}
\begin{tabular}{ccc}

\begin{tikzpicture}[->,>=stealth',shorten >=1pt,node distance=2cm,auto,main node/.style={circle,draw,align=center}]

\node[main node] (1) {\tiny{123}};
\node[main node] (2) [below of=1] {\tiny{213}};
\node[main node] (3) [right of=1] {\tiny{132}};
\node[main node] (4) [below of=3] {\tiny{312}};
\node[main node] (5) [right of=3] {\tiny{231}};
\node[main node] (6) [below of=5] {\tiny{321}};

\path
(1) edge node {} (3)
      edge [bend left=40] node  {} (5);
\path
(2) edge node {} (1);
\path
(5) edge node {} (6);
\path
(4) edge node {} (1);
\path
(3) edge node {} (6);
\path
(6) edge node {} (4)
      edge [bend left=40] node  {} (2);

\path
(2) edge [bend left=5] node  {} (3);
\path
(3) edge [bend left=5] node  {} (2);

\path
(3) edge [bend left=10] node  {} (4);
\path
(4) edge [bend left=10] node  {} (3);

\path
(4) edge [bend left=5] node  {} (5);
\path
(5) edge [bend left=5] node  {} (4);

\path
(2) edge [bend left=5] node  {} (5);
\path
(5) edge [bend left=5] node  {} (2);

\path
(1) edge [loop above] node {} (1);
\path
(6) edge [loop below] node {} (6);

\end{tikzpicture}

& \ \ \ &

\begin{tikzpicture}[node distance=2cm,auto,main node/.style={circle,draw,align=center}]

\node[main node] (1) {\tiny{123}};
\node[main node] (2) [below of=1] {\tiny{213}};
\node[main node] (3) [right of=1] {\tiny{132}};
\node[main node] (4) [below of=3] {\tiny{312}};
\node[main node] (5) [right of=3] {\tiny{231}};
\node[main node] (6) [below of=5] {\tiny{321}};

\path
(1) edge node {} (3)
      edge [bend left=40] node  {} (5);
\path
(2) edge node {} (1);
\path
(5) edge node {} (6);
\path
(4) edge node {} (1);
\path
(3) edge node {} (6);
\path
(6) edge node {} (4)
      edge [bend left=40] node  {} (2);

\path
(2) edge node  {} (3);

\path
(3) edge node  {} (4);

\path
(4) edge node  {} (5);

\path
(2) edge node  {} (5);

\end{tikzpicture}

\end{tabular}

\end{center}

\vspace{-0.8cm}
\caption{From left to right: the graphs $P(3)$ and $SP(3)$.}\label{graph-ex-2}
\end{figure}
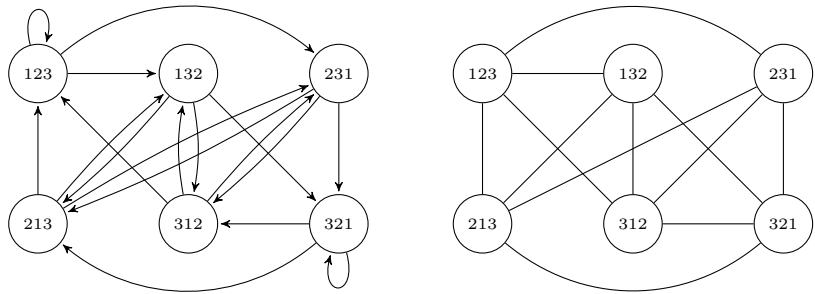

\subsection{Word-representability of graphs}

\noindent 
The literature contains a substantial body of research papers focused on the theory of word-representable graphs, as evidenced by references \cite{K17,KL15,KP2018} and related works. These graphs are of interest due to their connections to algebra, graph theory, computer science, combinatorics on words, and scheduling  \cite{KL15}. Notably, word-representable graphs extend the scope of several important graph classes, including {\em circle graphs}, {\em $3$-colorable graphs}, and {\em comparability graphs}. The ability to represent simplified de Bruijn graphs and simplified graphs of overlapping permutations using words could expand the range of potential applications for these graphs. This provides motivation for studying these graphs from the perspective of word-representability.

Two letters $x$ and $y$ alternate in a word $w$ if after deleting in $w$ all letters but the copies of $x$ and $y$ we either obtain a word $xyxy\cdots$ or a word $yxyx\cdots$ (of even or odd length).  A graph $G=(V,E)$ is {\em word-representable} if and only if there exists a word $w$
over the alphabet $V$ such that letters $x$ and $y$, $x\neq y$, alternate in $w$ if and only if $xy\in E$. 
The unique minimum (by the number of vertices) non-word-representable graph on 6 vertices is the wheel graph $W_5$, while there are 25 non-word-representable graphs on 7 vertices \cite{KL15}. 

An orientation of a graph is {\em semi-transitive} if it is acyclic, and for any directed path $v_0\rightarrow v_1\rightarrow \cdots \rightarrow v_k$ either there is no arc from $v_0$ to $v_k$, or $v_i\rightarrow v_j$ is an arc for all $0\leq i<j\leq k$. An induced subgraph on vertices $\{v_0,v_1,\ldots,v_k\}$ of an oriented graph is a {\em shortcut} if it is acyclic, non-transitive, and contains both the directed path $v_0\rightarrow v_1\rightarrow \cdots \rightarrow v_k$ and the arc $v_0\rightarrow v_k$, that is called the {\em shortcutting edge}. A semi-transitive orientation can then be alternatively defined as an acyclic shortcut-free orientation. A fundamental result in the area of word-representable graphs is the following theorem.

\begin{thm}[\cite{Halldorsson}]\label{semi-trans-thm} A graph is word-representable if and only if it admits a semi-transitive orientation. \end{thm} 

The following simple corollary of Theorem~\ref{semi-trans-thm}  is also instrumental for us.

\begin{thm}[\cite{Halldorsson}]\label{3-col-thm}
	Any $3$-colorable graph is word-representable.
\end{thm}

Directly related to our studies is the following theorem by Petyuk \cite{Petyuk}.

\begin{thm}[\cite{Petyuk}]\label{S0n2-thm}
	For positive integers $n$ and $k$ with $k\geq 3$, 
	\begin{wst}
		\item[\rm (i)] $S(n,2)=S_0(n,2)$ is word-representable;
		\item[\rm (ii)] $S(2,k)=S_0(2,k)$ is non-word-representable;
		\item[\rm (iii)] $S(3,k)=S_0(3,k)$ is non-word-representable.
	\end{wst}
\end{thm}

We will provide a sketch of the proof of  Theorem~\ref{S0n2-thm}(i) in Section~\ref{Pre-sec} in order to introduce our notation and color classes that will be used in this paper.

\subsection{Results in this paper}

\noindent In this paper, we continue the research on word-representability of the simplified de Bruijn graph initiated in \cite{Petyuk}, and extend the studies to the simplified graph of overlapping permutations. More specifically, the following conjecture  is stated in \cite{Petyuk}:
\begin{conj}\label{conj} $S(n,k)$ is non-word-representable for $n\geq 4$ and $k\geq 3$. \end{conj} 
\noindent Towards settling this conjecture it is natural to consider induced subgraphs of $S(n,k)$, such as $S_m(n,k)$, that have simpler structure but may still be non-word-representable. So, if $S_m(n,k')$ is non-word-representable for some $m, n, k'$ then $S_m(n,k)$ and hence $S(n,k)$, are non-word-representable for every $k\geq k'$.  

In order to study word-representability of $S_m(n,k)$, we use an {\em embedding technique} based on the well-known notion of  graph homomorphisms (see, e.g.~\cite{HN}). Let $G$ and $H$ be two graphs. Assume that there exists a mapping $f$ from $V(G)$ to $V(H)$ preserving adjacency (i.e.\ if $uv\in E(G)$ then $f(u)f(v)\in E(H)$). Then such a mapping $f$ is called a {\em homomorphism} (or just an {\em embedding}) from $G$ to $H$. For instance, for a $k$-colorable graph $G$, its $k$-coloring induces a homomorphism from $G$ to $K_k$ (the complete graph of order $k$). 

We found out that  homomorphisms can be useful for finding potential semi-transitive orientations of $G$. 
To the best of our knowledge, this is the first use of the homomorphisms in the area of word-representable graphs (although the proof of Theorem~\ref{3-col-thm} in \cite{Halldorsson} may be considered as a simplest application of this technique). Informally, the basic steps of the embedding technique (that will be indicated explicitly in the proofs of Theorem~\ref{thmspn} and Lemmas~\ref{lemo2n}--\ref{lemma-final}) are as follows: \\[-3mm]

\noindent
{\bf Step 1.} Given a graph $G$, find a word-representable graph $H$ such that there exists a homomorphism $f$ from $G$ to $H$. Based on a fixed acyclic orientation of $H$, orient all edges in $G$ as in $H$ (i.e.\ $u\rightarrow v$ in $G$ if $f(u)\rightarrow f(v)$ in the oriented $H$). Clearly, we obtain an acyclic orientation of $G$.\\[-3mm]

\noindent
{\bf Step 2.} Find all directed paths $f(v_0)\rightarrow f(v_1)\rightarrow \cdots \rightarrow f(v_k)$, $k\geq 3$, with $f(v_0) \rightarrow f(v_k)$ in the oriented $H$. We call such paths {\it shortcutting paths}. 
If no shortcutting paths exist in $H$, then the orientation of $G$ is semi-transitive, and we terminate the procedure. 
For instance, Theorem~\ref{3-col-thm} can be proved by embedding a 3-colorable graph into a triangle and orienting the triangle transitively. \\[-3mm]

\noindent
{\bf Step 3.} Analyze each shortcutting path found in Step 2. If none of them induces a shortcut in $G$  then we obtain a desired semi-transitive orientation of $G$.\\[-3mm]

\noindent
{\bf Remark.} We can find all directed paths in Step 2, for instance, in the following way. Let $M$ be the adjacency matrix of $H$, i.e.\  $a_{ij}=1$ if there is an arc from $v_i$ to $v_j$ in $H$ and $a_{ij}=0$ otherwise. Then $H$ contains a shortcutting path of length $p$  from $v_i$ to $v_j$ with the shortcutting edge $v_iv_j$ if and only if both matrices $M$ and its $p$-th power $M^p$ have positive elements in position $(i,j)$. 

Of course, Step~3 is not always applicable since there can be a homomorphism from a non-word-representable graph to a word-representable graph (for instance, a homomorphism from $W_5$ to $K_4$). However, sometimes the approach works providing elegant proofs of word-representability for quite complicated graphs. 

As an application of the embedding technique, we will prove the following two theorems.

\begin{thm}\label{thmspn}
	For any positive integer $n$, $SP(n)$ is word-representable.
\end{thm} 

Theorem~\ref{thmspn} is proved in Section~\ref{SP(n)-sec}.
In our proof of Theorem~\ref{thmspn}, the oriented $H$ found in Step~1 contains no shortcutting paths.

\begin{thm}\label{thmsnk}
	Let $m,n,k$ be positive integers.
	\begin{wst}
		\item[\rm (i)] For $n=2,3,4$ and $k\leq (n+1)m$, $S_m(n,k)$ is word-representable. Moreover, $S_m(2,k)$ is non-word-representable for $k\geq 3m+1$.
		\item[\rm (ii)]	For $n\geq 5$ and $k\leq 2mn$, $S_m(n,k)$ is word-representable.
	\end{wst}
\end{thm}

Theorem~\ref{thmsnk} is proved  in Section~\ref{Smnk-sec}  and in Appendix.

 In general, it is desirable that an orientation of $H$ fixed in {\bf Step~1} is semi-transitive, which is the case in the proofs of Theorem~\ref{thmspn} (mentioned above) and Lemmas~\ref{lemo2n}, \ref{lemma-6} and~\ref{lemma-final} (related to Theorem~\ref{thmsnk}(ii)). However, this is not a necessary condition as is demonstrated by us in the proof of Lemma~\ref{leme2n} in Section~\ref{even-Sec} (related to Theorem~\ref{thmsnk}(ii)). 

\section{Preliminaries}\label{Pre-sec}

In this section we provide simple (known) statements that are required for deriving our further results. The following lemma reveals relations among $S_m(n,k)$ for various $m$ and $k$.

\begin{lem} \label{obs}
	For any positive integers $m,n,k$ with $k\geq 2$, we have
	\begin{wst}
		\item[\rm (i)]$S_m(n,k)$ is isomorphic to a subgraph of $S_m(n,k+1)$;
		\item[\rm (ii)]$S_1(n,k)$ is isomorphic to a subgraph of $S_{m}(n, (k-1) m+1)$.
	\end{wst}
\end{lem} 
\begin{proof}
	It is clear that (i) holds. To prove (ii) consider the following embedding 
of $S_1(n,k)$ into $S_m(n,(k-1)m+1)$. For any vertex $\omega=x_1x_2\ldots x_n$ in $S_1(n,k)$, let $f(\omega)$ be the word of length $n$ obtained by multiplying each letter of $\omega$ by $m$, that is,
	$f(\omega)=(m x_1)(m x_2)\ldots(m x_n).$ Since $x_i\leq k-1$, $f(\omega)$ is a vertex in $S_m(n,(k-1)m+1)$. Also, $\omega\omega'\in E(S_1(n,k))$ if and only if $f(\omega)f(\omega')\in E(S_m(n,(k-1)m+1))$, which implies that (ii) holds.
\end{proof}

We also need a proof of $3$-colorability of $S_0(n,2)$ that was found by Petyuk in \cite{Petyuk}. Here we rewrite the original proof in a more convenient notation that will be of use for us later on. \\[-3mm]

\noindent
{\bf A $3$-colorability of $S_0(n,2)$.} We first introduce the following types of words:
	\begin{wst}
		\item[\footnotesize $\bullet$] $e_1$ (resp., $e_0)$ represents words consisting of a positive {\em even} number of 1s (resp., 0s) only; 
		\item[\footnotesize $\bullet$] $o_1$ (resp., $o_0)$ represents words  consisting of an {\em odd} number of 1s (resp., 0s) only;
		\item[\footnotesize $\bullet$] $a_1$ (resp., $a_0)$ represents {\em non-empty} words {\em beginning} with 1 (resp., 0);
		\item[\footnotesize $\bullet$] $b_1$ (resp., $b_0)$ represents either words {\em beginning} with 1 (resp., 0) or an {\em empty} word.
	\end{wst}

Then the vertices in $S_0(n,2)$ may be represented (in different ways) using this notation. The following observation is easy to verify.

\begin{Obs}\label{obs2}
	Let $w$ be a word representing a vertex in $S_0(n,2)$ and let the word $w'$ be obtained from $w$ by removing the first letter. Then
\begin{wst}
		\item[\rm (i)] if $w$ has the form $e_1b_0$ (resp., $e_0b_1$) then $w'$ has the form $o_1b_0$ (resp., $o_0b_1$);
		\item[\rm (ii)] if $w$ has the form $o_1b_0$ (resp., $o_0b_1$) then $w'$ has either the form $e_1b_0$ (resp., $e_0b_1$) or the form $b_0$ (resp., $b_1$).
	\end{wst}
\end{Obs}

Consider two cases.

 {\noindent \bf Case 1.} $n$ is even. In this case, each vertex of $S_0(n,2)$ can be represented by exactly one of $\{e_0b_1,e_1b_0,o_1e_0a_1,o_1o_0b_1,o_0e_1a_0,o_0o_1b_0\}$, which is called the {\em form} of this vertex. Then we can color the vertices in $S_0(n,2)$ based on their forms as follows:
    \begin{align*}
    	\text{\bf Red: } e_0b_1,e_1b_0; &&\text{\bf Blue: } o_1e_0a_1, o_1o_0b_1; &&
    	\text{\bf Green: } o_0e_1a_0, o_0o_1b_0.
    \end{align*}
  
    {\noindent \bf Case 2.} $n$ is odd. Similarly, we color the vertices in $S_0(n,2)$ based on their forms as follows:
    \begin{align*}
    	\text{\bf Red: } e_0a_1, e_1a_0; &&\text{\bf Blue: } o_1, o_1o_0a_1, o_1e_0b_1; &&
    	\text{\bf Green: } o_0, o_0o_1a_0, o_0e_1b_0.
    \end{align*}

Using Observation~\ref{obs2}, it is not difficult to verify that in both cases there are no monochromatic edges (note that the vertices of the forms $e_1, e_0, o_1,$ and $o_0$ are unique and $S_0(n,2)$ has no loops), and thus we obtain a proper 3-coloring of $S_0(n,2)$.

\section{Word-representability of $SP(n)$}\label{SP(n)-sec}

   In this section, we give a proof of Theorem~\ref{thmspn} using the embedding technique. Note that $SP(1)$ is a single vertex while $SP(2)$ is a single edge. Both of these graphs are word-representable. Next, we consider the case of $n\geq 3$.

{\bf Step 1.}  Let $H=S_0(n-1,2)$. For any vertex $\omega=x_1x_2\ldots x_n$ in $SP(n)$, let $\tau(\omega)=y_1\ldots y_{n-1}$ be such that $y_i=0$ if $x_i>x_{i+1}$ and $y_i=1$ otherwise. Then $\tau(w)$ is a mapping from $V(SP(n))$ to $V(S_0(n-1,2))$. Let us show that $\tau$ is a homomorphism. 	
   Assume that $\omega\omega'$ is an edge in $E(SP(n))$, where $\omega=x_1x_2\ldots x_n$, $\omega'=x'_1x'_2\ldots x'_n$ and $\omega\rightarrow \omega'$ is an arc in $P(n)$. 

   If $\tau(\omega)=y_1y_2\ldots y_{n-1}$, then $\tau(\omega')=y_2\ldots y_{n-1}y_n$ by definition of $\tau$. Thus, if $\tau(\omega)\neq \tau(\omega')$, then there exists an edge between $\tau(\omega)$ and $\tau(\omega')$. However, $\tau(\omega)=\tau(\omega')$ implies that $y_1=\cdots=y_n\in \{0,1\}$ so that $\omega=\omega'\in \{123\ldots n,\ n(n-1)(n-2)\ldots 1\}$, which is a contradiction with $\omega\neq\omega'$. So $\tau(\omega)$ is adjacent to $\tau(\omega')$ in $S_0(n-1,2)$, as desired.
	
Let $R\cup B\cup G$ be a 3-coloring partition of $V(S_0(n-1,2))$. Orient $S_0(n-1,2)$ so that any vertex in $R$ is a source and any vertex in $G$ is a sink. 

{\bf Step 2. }Orient all edges in $SP(n)$ as in $S_0(n-1,2)$ (i.e., $u\rightarrow v$ in $SP(n)$ if $\tau(u)\rightarrow \tau(v)$ in the oriented $S_0(n-1,2)$).
   	
Since there are no shortcutting paths in the oriented $S_0(n-1,2)$, we obtain a desired semi-transitive orientation of $SP(n)$ ({\bf Step~3} is not needed). By Theorem~\ref{semi-trans-thm}, $SP(n)$ is word-representable.
   
\section{Word-representability of $S_m(n,k)$}\label{Smnk-sec}

\subsection{The case of $n=2$}

\begin{lem}
	If $m\geq 1$ and $k\leq 3m$ then $S_m(2,k)$ is word-representable. 
\end{lem}
\begin{proof}
For all pairs of adjacent vertices $x_1x_2$ and $y_1y_2$ in $S_m(2,k)$, orient the edge by the lexicographical order (i.e.\  we have $x_1x_2\rightarrow y_1y_2$ if $x_1x_2$ is lexicographically smaller than $y_1y_2$). It is clear that this orientation is acyclic. 
We claim that the oriented graph contains no directed path of length 3, which implies that there are no shortcuts in this orientation and thus $S_m(2,k)$ is word-representable.
	
	Indeed, suppose that there is a directed edge from $x_1x_2$ to $y_1y_2$ in $S_m(2,k)$. Then $x_1x_2$ is lexicographically smaller than $y_1y_2$, i.e.\  $y_1\geq x_1$. Moreover, either  $x_{2}=y_{1}$ or $x_1=y_{2}$. If $x_{2}=y_{1}$ then $x_2\geq x_1$. By definition of $S_m(2,k)$, we have $x_2-x_1\geq m$, and so, $y_1\geq x_1+m$. If $y_{2}=x_{1}$ then, by a similar argument, $y_1 \geq y_2+m =x_1+m$. Hence, in both cases, $y_1\geq x_1+m$. For $\{\omega_1,\omega_2,\omega_3,\omega_4\}\subseteq V(S_m(2,k))$, suppose that there is a directed path $\omega_1\rightarrow \omega_2 \rightarrow \omega_3\rightarrow \omega_4$ of length 3. Then, the first letter of $\omega_4$ is at least $3m$. However, since $k\leq 3m$, any letter in $A$ is at most $3m-1$, a contradiction. So there are no directed paths of length 3, as desired. 
\end{proof}

To finish the study of the case of $n=2$, it remains to prove that $S_m(2,k)$ is non-word-representable for $k\geq 3m+1$. By Lemma~\ref{obs}(i), it is sufficient to show that $S_m(2,3m+1)$ is non-word-representable, and by Lemma~\ref{obs}(ii), our aim is to prove that $S_1(2,4)$ is non-word-representable. Since this proof is a tedious case-analysis requiring certain special encodings, it is moved to Appendix.

\subsection{A bipartite subgraph of $S_m(n,k)$}
In this section, we introduce some notations necessary for our proofs in the rest of the paper. Let
 $$A^{n}_{m}(k,<):=\{x_1x_2\ldots x_n:x_i\leq x_{i+1}-m,\ 1\leq i\leq n-1\}$$ and $$A^{n}_{m}(k,>):=\{x_1x_2\ldots x_n:x_i\geq x_{i+1}+m,\ 1\leq i\leq n-1\}$$ be, respectively, the sets of increasing  and decreasing words in $A^n_m(k)$. We denote the subgraphs of $S_m(n,k)$ induced by $A^{n}_{m}(k,<)$ and $A^{n}_{m}(k,>)$ by $S^{<}_m(n,k)$ and $S^{>}_m(n,k)$, respectively. Then we claim that both $S^{<}_m(n,k)$ and $S^{>}_m(n,k)$ are triangle-free. Indeed, suppose $\omega_1\omega_2\omega_3$ is a triangle in $S^{<}_m(n,k)$. Let $\omega_1=x_1x_2\ldots x_n$ with $x_i\leq x_{i+1}-m$. Since $\omega_2$ is adjacent to $\omega_3$, their first letters must be different. Because $\omega_1\omega_2$ and $\omega_1\omega_3$ are edges in $S^{<}_m(n,k)$, $\{\omega_2,\omega_3\}=\{x_0x_1\ldots x_{n-1},x_2\ldots x_n x_{n+1}\}$ for some $x_0\leq x_1-m$ and $x_{n+1}\geq x_n+m$, which is a contradiction with $\omega_2\omega_3\in E(S^{<}_m(n,k))$.

    \begin{lem}\label{lem2<}
		$S^{<}_m(2,4m)$ is bipartite.
	\end{lem}
	\begin{proof}
		Suppose that $C$ is a shortest odd cycle in $S^{<}_m(2,4m)$. Then $C$ is chord-free. Orient the edges in $C$ in the same way as they are oriented in de Bruijn graph $B(2,4m)$ and denote the obtained directed graph by $\mathop{C}\limits ^{\rightarrow}$. 
Since $k=4m$, the longest directed path in $B(2,4m)$ (and hence in $\mathop{C}\limits ^{\rightarrow}$) is of length at most~2. 
		
		If there are no directed paths of length 2 in $\mathop{C}\limits ^{\rightarrow}$, then $\mathop{C}\limits ^{\rightarrow}$ is clearly an even cycle, a contradiction.
		
		Suppose that $xy\rightarrow yz\rightarrow zw$ is a directed path of length 2 in $\mathop{C}\limits ^{\rightarrow}$. Then $xy$ and $zw$ are a source and a sink in $\mathop{C}\limits ^{\rightarrow}$, respectively. If the other neighbour of $xy$ in $C$, say $ya$, is a sink, then $C$ contains the following path (arrows indicate the directions of the edges in $\mathop{C}\limits ^{\rightarrow}$):
		$$by\rightarrow ya\leftarrow xy\rightarrow yz\rightarrow zw.$$
		However, in this case, there is a chord $by\rightarrow yz$ in $C$, which is a contradiction. Thus, $xy$, or generally the head of any directed path of length 2, is not adjacent to any sink. Similarly, the tail of any directed path of length 2 is not adjacent to any source. This implies that $\mathop{C}\limits ^{\rightarrow}$ comprises of a series of alternate directed paths of length 2. Then $C$ is of even length, which is a contradiction.
		
		Therefore $S^{<}_m(2,4m)$ is bipartite, as desired. \end{proof}

Assume that $A$ is a subset of $A_m^n(k)$. For any word $\omega$ in $A^{n-1}_m(k)$, the {\em cluster} with respect to $\omega$ is the subset of $A$ comprised of all words that begin or end with $\omega$. A cluster is non-trivial if it contains at least two vertices.
	The {\em cluster graph} $\mathcal{C}(A)$ of $A$ is the graph whose vertex set is the set of all non-trivial clusters of $A$. There is an edge between two clusters in $\mathcal{C}(A)$ if and only if they have a common vertex in $A$. Note that $\mathcal{C}(A)$ is a subgraph of $S_m(n-1,k)$ provided $n\geq 2$. For example, if $A=\{0123,1234,2123,2345\}$, then the cluster with respect to $123$ is $\{0123,1234,2123\}$, while the cluster with respect to $012$ is $\{0123\}$, and the cluster with respect to $124$ is the empty set. In this case, $\mathcal{C}(A)$ is a single edge with vertex set $\{123,234\}$.

	\begin{lem}\label{lem<}
		If $n\geq 2$ and $k\leq 2mn$ then $S^{<}_m(n,k)$ is bipartite. 
	\end{lem}
	\begin{proof}
		Since $S^{<}_m(n,k)$ is a subgraph of $S^{<}_m(n,k+1)$, it is sufficient to prove that $S^{<}_m(n,2mn)$ is bipartite. We apply the induction on $n$. By Lemma \ref{lem2<}, $S^{<}_m(n,2mn)$ is bipartite for $n=2$. Assume that $S^{<}_m(n,2mn)$ is bipartite. Now we consider the graph $S^{<}_m(n+1,2m(n+1))$.
		
		Suppose that $C$ is a shortest (and hence chord-free) odd cycle in $S^{<}_m(n+1,2m(n+1))$ and the length of $C$ is $2t+1$. Direct the edges in $C$ in the same way as they are oriented in de Bruijn graph $B(n+1,2m(n+1))$ and obtain the directed graph $\mathop{C}\limits ^{\rightarrow}$. Let $\omega_1'\omega_1\omega_2\omega_2'$ be four consecutive vertices in $C$. Without loss of generality, assume that $\omega_2$ starts with $\omega$ and ends with $\omega'$, where $\omega$ and $\omega'$ are both in $A^n_m(k,<)$. Based on the orientation of $\mathop{C}\limits ^{\rightarrow}$, we have the following observations.
		\begin{wst}
			\item[\footnotesize $\bullet$] At most one of $\{\omega_1,\omega_2\}$ is a sink or a source. Otherwise, without loss of generality, suppose that $\omega_1$ is a sink and $\omega_2$ is a source. Then both $\omega_1$ and $\omega_2'$ begin with $\omega'$, while $\omega_1'$ ends with $\omega'$. This implies that $\omega_1'$ is adjacent to $\omega_2'$ in $S^{<}_m(n+1,2m(n+1))$, which is a contradiction. 
			\item[\footnotesize $\bullet$] If $\omega_2$ is a sink then $\{\omega_1,\omega_2,\omega_2'\}$ in $C$ is the cluster with respect to $\omega$.
			\item[\footnotesize $\bullet$] If $\omega_2$ is a source then $\{\omega_1,\omega_2,\omega_2'\}$ is the cluster with respect to $\omega'$.
			\item[\footnotesize $\bullet$] If none of $\omega_1$ and $\omega_2$ is a sink or a source and $\omega_1\rightarrow\omega_2$ then $\{\omega_1,\omega_2\}$ is the cluster with respect to $\omega$. 
			\item[\footnotesize $\bullet$] If none of $\omega_1$ and $\omega_2$ is a sink or a source and $\omega_2\rightarrow\omega_1$ then $\{\omega_1,\omega_2\}$ is the cluster with respect to $\omega'$. 	
		\end{wst}
		
Let $s$ be the total number of sinks and sources in $\mathop{C}\limits ^{\rightarrow}$. Then $s$ must be even. Combining with the above observations, the cluster derived from consecutive vertices in $C$ has size 3 if it contains a sink or a source and has size 2 otherwise. Thus, the cluster graph $\mathcal{C}(V(C))$ also contains a cycle $C'$ of length $2t+1-s$ (recall that the length of $C$ is $2t+1$), which is still odd. We give an illustration of $\mathop{C}\limits ^{\rightarrow}$ and $\mathcal{C}(V(C))$ for the graph $S_1^<(4,8)$ in Figure \ref{fig2}. In this figure, since $t=7$ and $s=4$, the length of $C'$ is $11$. 
		
		Now, let $x$ be the smallest letter occurring in any word in $V(C)$. For any letter $x_0\in \{x,x+1,\ldots,x+m-1\}$, if $\omega=x_0x_1\ldots x_{n}$ is a word in $V(C)$ beginning with $x_0$, then $x_1\geq x_0+m\geq x+m$, and $\omega$ must be a source in $\mathop{C}\limits ^{\rightarrow}$. In this case, $\omega$ is in the cluster with respect to $x_1\ldots x_{n}$, which does not contain $x_0$. Thus, no word in $C'$ contains any letter in $\{x,x+1,\ldots,x+m-1\}$. Similarly, no word in $C'$ contains any letter in $\{y,y-1,\ldots,y-m+1\}$, where $y$ is the largest letter in $V(C)$. Then, the alphabet comprised of all letters occurring in $V(C')$ has size at most $2m(n+1)-2m=2mn$.

		Note that the cluster graph $\mathcal{C}(V(C))$ is a subgraph of $S^{<}_m(n,2m(n+1))$. Then, $C'$ is isomorphic to a subgraph of $S^{<}_m(n,2mn)$, which is a contradiction to the induction hypothesis. Hence, any cycle in $S^{<}_m(n+1,2m(n+1))$ is an even cycle and $S^{<}_m(n+1,2m(n+1))$ is bipartite, as desired. \end{proof}

    \begin{figure}
    	\centering
    	\includegraphics[width=135mm]{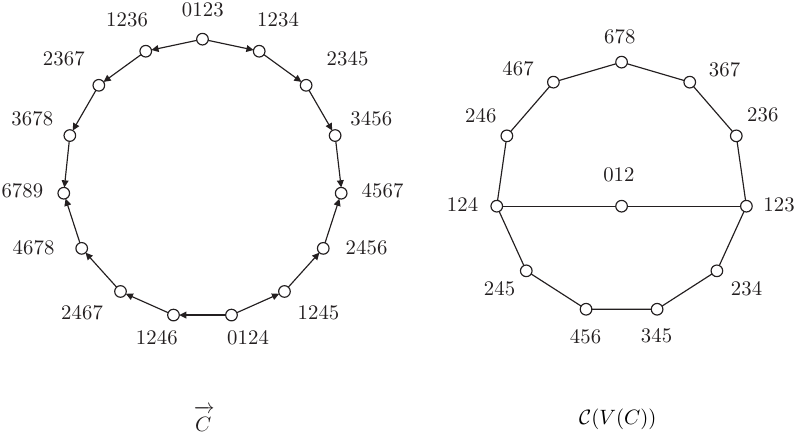}
    	\caption{The oriented cycle $\mathop{C}\limits ^{\rightarrow}$ and the cluster graph $\mathcal{C}(V(C))$ in $S_1^<(4,8)$.}
    	\label{fig2}
    \end{figure}

    {\noindent \bf Remark.} In Lemmas~\ref{lem2<} and~\ref{lem<}, the bounds on $k$ are sharp because 
$S^<_{1}(2,5)$ contains an odd cycle
    $$01-12-23-34-13-01$$    
and $S^<_{1}(3,7)$ contains an odd cycle
    $$012-123-234-345-456-245-124-012.$$

\subsection{Word-representability of $S_m(n,k)$ for odd $n\geq 3$}
First, we show that Theorem~\ref{thmsnk}(ii) holds for every odd $n\geq 5$.

\begin{lem}\label{lemo2n}
	For any odd $n\geq 5$ and $k\leq 2mn$, $S_m(n,k)$ is word-representable.
\end{lem}
\begin{proof}
	By Lemma \ref{lem<}, $S^{<}_m(n,k)$ is bipartite, the parts of which are denoted by $R_1$ and $R_2$.
    Let $\overline{R}_1\cup \overline{R}_2$ be a partition of $A^n_m(k,>)$, where 
	$$\overline{R}_i=\{x_nx_{n-1}\ldots x_1\ |\ x_1x_2\ldots x_n \in R_i\} \text{ for } i=1,2.$$
	Then $S^{>}_m(n,k)$ is also bipartite with parts $\overline{R}_1\cup \overline{R}_2$.
       
    For each vertex $\omega=x_1x_2\ldots x_n$ in $S_m(n,k)$, put $\tau(\omega)=y_1\ldots y_{n-1}$ where $y_i=0$ if $x_i>x_{i+1}$ and $y_i=1$ otherwise. Then $\tau$ is a mapping from $S_m(n,k)$ to $S_0(n-1,2)$. Also, recall that $S_0(n-1,2)$ is 3-colorable, whose vertices can be colored based on its forms introduced in Section~\ref{Pre-sec}:
        \begin{align*}
        	\text{\bf Red: } e_0, e_1, e_0a_1, e_1a_0; &&\text{\bf Blue: } o_1e_0a_1, o_1o_0b_1; &&
        	\text{\bf Green: } o_0e_1a_0, o_0o_1b_0.
        \end{align*}

Note that a word of the form $e_0b_1$ (resp. $e_1b_0$) can be either $e_0$ or $e_0a_1$ (resp. $e_1$ or $e_1a_0$). Here we separate these two cases.
By the definition of   $A^n_m(k,<)$ and $A^n_m(k,>)$, the form of $\tau(\omega)$ is $e_1$ (resp., $e_0$) if and only if $\omega$ is a vertex in $R_1\cup R_2$ (resp., $\overline{R}_1\cup \overline{R}_2$). 
    
        \begin{figure}
        	\begin{center}        			
        		\begin{tikzpicture}[->,>=stealth',shorten >=1pt,node distance=2cm,auto,main node/.style={circle,draw,align=center}]
        			
        			\node[main node] (1) {$R'$};
        			\node[main node] (2) [left of=1] {$B$};
        			\node[main node] (3) [above of=1] {\footnotesize  $\overline{R}_2$};
        			\node[main node] (4) [below of=1] {\footnotesize  $R_2$};
        			\node[main node] (5) [right of=1] {$G$};
        			
        			\path
        			(1) edge node {} (2);
        			\path
        			(1) edge node {} (3);
        			\path
        			(1) edge node {} (4);
        			\path
        			(1) edge node {} (5);
        			\path
        			(2) edge [bend left=27] node  {} (5);
        			\path
        			(2) edge node {} (3);
        			\path
        			(2) edge node {} (4);	
        			\path
        			(3) edge node {} (5);
        			\path
        			(4) edge node {} (5);			        
        		\end{tikzpicture}		
        	\end{center}
        	\caption{A graph $H$ and its orientation.}
        	\label{fig3}
        \end{figure}
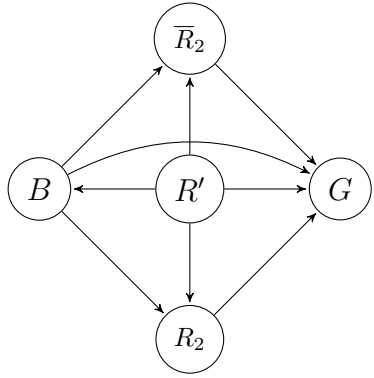
{\bf  Step~1.}		Let $H$ be the underlying (undirected) graph with vertex set $\{R_2,\overline{R}_2,R',B,G\}$ in Figure~\ref{fig3}. 
Consider the following mapping $f$ from $S_m(n,k)$ to $H$. 
		\begin{wst}
			\item[\footnotesize $\bullet$] If $\omega\in R_2$ then $f(\omega)=R_2$.
			\item[\footnotesize $\bullet$] If $\omega\in \overline{R}_2$ then $f(\omega)=\overline{R}_2$.
            \item[\footnotesize $\bullet$] If $\omega\in R_1\cup \overline{R}_1$ or the form of $\tau(\omega)$ is one of $\{e_0a_1, e_1a_0\}$ then $f(\omega)=R'$.            \item[\footnotesize $\bullet$] If the form of $\tau(\omega)$ is one of $\{o_1e_0a_1, o_1o_0b_1\}$ then $f(\omega)=B$.
            \item[\footnotesize $\bullet$] If the form of $\tau(\omega)$ is one of $\{o_0e_1a_0, o_0o_1b_0\}$ then $f(\omega)=G$.		
            \end{wst}
Clearly, all sets of vertices of $S_m(n,k)$ mapped into the same vertex of $H$ are independent.
Since there are  no edges between $R_2$ and $\overline{R}_2$, $f$ is a homomorphism from $S_m(n,k)$ to $H$. Then direct $H$ as in Figure~\ref{fig3}.

{\bf Step~2.} Note that the orientation of 
$H$ is acyclic and contains only two  shortcutting paths of length at least 3:  
		\begin{align*}
			R'\rightarrow B\rightarrow R_2\rightarrow G;&&
			R'\rightarrow B\rightarrow \overline{R}_2\rightarrow G.
		\end{align*}

Now direct each edge in $S_m(n,k)$ as in $H$ (i.e.\  $u\rightarrow v$ in $S_m(n,k)$ if $f(u)\rightarrow f(v)$ in $H$). 
 It remains to prove that the orientation of $S_m(n,k)$ contains no shortcuts and hence it is semi-transitive.
		
{\bf Step~3.}		Suppose that there exists a directed path $\omega_{R'}\rightarrow\omega_B\rightarrow\omega_{R_2}\rightarrow\omega_G$ with $f(\omega_{X})=X$ for all $X\in \{R',B,R_2,G\}$, and the path introduces a potential shortcutting edge, that is, $\omega_{R'}\rightarrow\omega_G$.
		
		Let $\omega_{R_2}=x_1x_2\ldots x_n$ with $x_i\leq x_{i+1}-m$. 
Since $n$ is odd and the forms of $\tau(\omega_G)$ and $\tau (\omega_B)$ start with $o_0$ and $o_1$ respectively, we have
$\omega_B=x_2\ldots x_{n-1}x_ny$ and $\omega_G=zx_1x_2\ldots x_{n-1}$, where $y\leq x_n-m$ and $z\geq x_1+m$. 

Since $\omega_{R'}$ is adjacent to $\omega_{G}$, we have either $\omega_{R'}=wzx_1\ldots x_{n-2}$ or $\omega_{R'}=x_1\ldots x_{n-1}w$ for some $w$. Since $n\ge 5$, in the former case the form of $\tau(\omega_{R'})$ is $e_0e_1$, or $o_1o_0e_1$, while
in the latter case, it is either $o_1o_0$ or $e_1$. In either case, the form of $\tau(\omega_{R'})$ belongs to the set $F_1=\{o_1o_0, o_1o_0e_1, e_0e_1, e_1 \}$. On the other hand, $\omega_{R'}$ is adjacent to $\omega_{B}$, and hence either $\omega_{R'}=ux_2\ldots x_n$ or $\omega_{R'}=x_3\ldots x_{n}yu$ for some $u$. In the former case, the form of $\tau(\omega_{R'})$ is either $e_1$ or $o_0o_1$, while in the latter case it may be $e_1e_0$ or $e_1o_0o_1$.
Then, the form of $\tau(\omega_{R'})$ belongs to the set $F_2=\{o_0o_1, e_1e_0, e_1o_0o_1, e_1 \}$. Clearly, $F_1\cap F_2= \{ e_1\}$. But if the form of $f(\omega_{R'})$ is $e_1$, then 
$\omega_{R'}=x_1\ldots x_{n-1}w=ux_2\ldots x_n=x_1x_2\ldots x_n=\omega_{R_2}$, a contradiction.

So, there exists no shortcut induced by $\{\omega_{R'},\omega_B,\omega_{R_2},\omega_G\}$. Similarly, there are no shortcuts induced by $\{\omega_{R'},\omega_B,\omega_{\overline{R}_2},\omega_G\}$. Hence, the orientation of $S_m(n,k)$ is shortcut-free and therefore it is semi-transitive.  By Theorem~\ref{semi-trans-thm}, $S_m(n,k)$ is word-representable for all odd $n\geq 5$ and $k\leq 2mn$. \end{proof}

We prove the remaining case of $n=3$ in the following lemma.

\begin{lem}\label{lemma-6}
	For $k\leq 4m$, $S_m(3,k)$ is word-representable.
\end{lem}
\begin{proof}
	By Lemma~\ref{obs}, it is sufficient to show that $S_m(3,4m)$ is word-representable.
	Let $A$, $B$ and $C$ be the following subsets of $A^3_m(4m,<)$:
	\begin{align*}
	A&:=\{x_1x_2x_3\in A^3_m(4m,<):0\leq x_1\leq m-1 \text{ and } 3m\leq x_3\leq 4m-1\};\\
	B&:=\{x_1x_2x_3\in A^3_m(4m,<):m\leq x_1\leq 2m-1 \text{ and } 3m\leq x_3\leq 4m-1\};\\
	C&:=\{x_1x_2x_3\in A^3_m(4m,<):0\leq x_1\leq m-1 \text{ and } 2m\leq x_3\leq 3m-1\}.
    \end{align*}
    Since $x_3-x_1\geq 2m$ for $x_1x_2x_3\in A^3_m(4m,<)$, $A\cup B\cup C$ is a partition of $A^3_m(4m,<)$. Note that any vertex in $A$ is an isolated vertex in $S^<_m(3,4m)$ and all edges in $S^<_m(3,4m)$ are between $B$ and $C$. Let $R_1=A\cup B$ and $R_2=C$. Then $S_m^<(3,4m)$ is bipartite with parts $R_1\cup R_2$. Let $\overline{R}_1\cup \overline{R}_2$ be a partition of $A^n_m(4m,>)$, where 
    $$\overline{R}_i=\{x_3x_2x_1\ |\ x_1x_2x_3 \in R_i\} \text{ for } i=1,2.$$ 
    Then, $S^{>}_m(3,4m)$ is also bipartite with parts $\overline{R}_1\cup \overline{R}_2$.

{\bf Step~1.} Define a mapping from $S_m(3,4m)$ to the graph $H$ in Figure~\ref{fig3} as follows.
Let $w=x_1x_2x_3$ be a vertex of $S_m(3,4m)$.
\begin{wst}
			\item[\footnotesize $\bullet$] If $\omega\in R_2$, then $f(\omega)=R_2$.
			\item[\footnotesize $\bullet$] If $\omega\in \overline{R}_2$, then $f(\omega)=\overline{R}_2$.
            \item[\footnotesize $\bullet$] If $\omega\in R_1\cup \overline{R}_1$  then $f(\omega)=R'$.
            \item[\footnotesize $\bullet$] If $x_2\ge x_1+m$ and $x_2\ge x_3+m$, then $f(\omega)=B$. 
            \item[\footnotesize $\bullet$] If $x_2\le x_1-m$ and $x_2\le x_3-m$, then $f(\omega)=G$. 
		\end{wst}

{\bf Step~2.}	Direct $S_m(3,4m)$ as in $H$. Again, note that 
	$H$ contains only two shortcutting paths of length at least 3:  
	\begin{align*}
		R'\rightarrow B\rightarrow R_2\rightarrow G;&&
		R'\rightarrow B\rightarrow \overline{R}_2\rightarrow G.
	\end{align*}

{\bf Step~3.}    Suppose that there exists a directed path $\omega_{R'}\rightarrow\omega_B\rightarrow\omega_{R_2}\rightarrow\omega_G$ with $f(\omega_{X})=X$ for all $X\in \{R',B,R_2,G\}$, and the path introduces a potential shortcutting edge, that is, $\omega_{R'}\rightarrow\omega_G$.
    Assume that $\omega_{R_2}=x_1x_2x_3$ with $x_1+m\leq x_2\leq x_3-m$. According to the forms of $\omega_{B}$ and $\omega_{G}$, we have $\omega_B = x_2x_3z$ and $\omega_G = yx_1x_2$ where $x_1+m\le y$ and $z \le x_3-m$.
Since $\omega_{R'}$ is adjacent to $\omega_{B}$, we have $\omega_{R'}=vx_2x_3$ for some $v\leq x_2-m$ or $\omega_{R'}=x_3zv$ for some $v\le z-m$. 
Since $\omega_{R'}$ is adjacent to $\omega_{G}$, we have $\omega_{R'}=x_1x_2u$ for some  $u \ge x_2+m$ or $\omega_{R'}=uyx_1$ for some $u\geq y+m$. 
So, two cases are possible. If $\omega_{R'}=x_1x_2u=vx_2x_3$ then $v=x_1, u=x_3$ and $\omega_{R'}=x_1x_2x_3=\omega_{R_2},$ a contradiction.  
 If $\omega_{R'}=uyx_1=x_3zv$ then $v=x_1, u=x_3$ and $z=y,$ i.e.\  $\omega_{R'}=x_3yx_1$. 
Note  that $\omega_{R'}\in \overline{R}_1$  and hence $x_1yx_3$ is a word in $R_1$. However, since $\omega_{R_2}=x_1x_2x_3\in R_2$, $0\leq x_1\leq m-1$ and $2m\leq x_3\leq 3m-1$. This implies $x_1yx_3$ should be in $R_2$ but not in $R_1$, a contradiction.
    
    So, there are no shortcuts induced by $\{\omega_{R'},\omega_B,\omega_{R_2},\omega_G\}$. Similarly, there are no shortcuts induced by $\{\omega_{R'},\omega_B,\omega_{\overline{R}_2},\omega_G\}$. Hence, the orientation of $S_m(3,4m)$ is shortcut-free and therefore it is semi-transitive.  By Theorem~\ref{semi-trans-thm}, $S_m(3,4m)$ is word-representable, as desired. \end{proof}

\subsection{Word-representability of $S_m(n,k)$ for even $n\geq 4$}\label{even-Sec}
First, we show that Theorem~\ref{thmsnk}(ii) holds for any even $n\geq 6$.

\begin{lem}\label{leme2n}
	For even $n\geq 6$ and $k\leq 2mn$, $S_m(n,k)$ is word-representable.
\end{lem}
\begin{proof}
	By Lemma \ref{lem<}, $S^{<}_m(n,k)$ is bipartite, and we denote its parts by $B_1$ and $B_2$. Let $G_1\cup G_2$ be a partition of $A^n_m(k,>)$, where 
	$$G_i=\{x_nx_{n-1}\ldots x_1\ |\ x_1x_2\ldots x_n \in B_i\} \text{ for } i=1,2.$$
	Then $S^{>}_m(n,k)$ is also bipartite with parts $G_1\cup G_2$.
		
    For a vertex $\omega=x_1x_2\ldots x_n$ in $S_m(n,k)$, let $\tau(\omega)=y_1\ldots y_{n-1}$ where $y_i=0$ if $x_i>x_{i+1}$ and $y_i=1$ otherwise. Then $\tau$ is a mapping from $S_m(n,k)$ to $S_0(n-1,2)$. Recall that in the case of even $n$ the color classes of $S_0(n-1,2)$ are as follows:
    \begin{align*}
       \text{\bf Red: } e_0a_1, e_1a_0; &&\text{\bf Blue: }  o_1, o_1o_0a_1,  o_1e_0b_1; &&
       \text{\bf Green: } o_0, o_0o_1a_0, o_0e_1b_0.
    \end{align*}
Note that the form of $\tau(\omega)$ is $o_1$ (resp., $o_0$) if and only if $\omega$ is a vertex in $B_1\cup B_2$ (resp., $G_1\cup G_2$). 

        \begin{figure}
        	\begin{center}        			
        		\begin{tikzpicture}[->,>=stealth',shorten >=1pt,node distance=2cm,auto,main node/.style={circle,draw,align=center}]        			

        			\node[main node] (3) {\small $B_3$};
        			\node[main node] (2) [left of=3,xshift=-0.8284cm]{\small $B_1$};
        			\node[main node] (4) [right of=3,xshift=0.8284cm]{\small$G_2$};
        			\node[main node] (5) [below left of=3]{\small$G_4$};
        			\node[main node] (6) [below right of=3]{\small$B_4$};
        			\node[main node] (1) [right of=6,xshift=2cm] {\small $R_1$};
        			\node[main node] (10) [left of=5,xshift=-2cm] {\small $R_2$};
        			\node[main node] (7) [below left of=5]{\small$B_2$};
        			\node[main node] (8) [below right of=5]{\small$G_3$};
        			\node[main node] (9) [below right of=6]{\small$G_1$};
        			
        			\path
        			(1) edge [bend right=90] node {} (5);
        			\path
        			(4) edge node {} (1);
        			\path
        			(1) edge node {} (6);
        			\path
        			(1) edge node {} (9);
        			\path
        			(1) edge [bend left=60] node {} (8);
        			\path
        			(6) edge [bend left=90] node {} (10);
        			\path
        			(10) edge node {} (7);
        			\path
        			(2) edge node {} (10);
        			\path
        			(5) edge node {} (10);
        			\path
        			(3) edge [bend right=60] node {} (10);
        			\path
        			(2) edge node {} (7);
        			\path
        			(4) edge node {} (9);
        			\path
        			(5) edge node {} (2);
        			\path
        			(5) edge node {} (3);
        			\path
        			(5) edge node {} (7);
        			\path
        			(4) edge node {} (6);
        			\path
        			(8) edge node {} (6);
        			\path
        			(9) edge node {} (6);
        			\path
        			(8) edge node {} (3);

        		\end{tikzpicture}		
        	\end{center}
        	\caption{A graph $H$ and its orientation.}
        	\label{fig4}
        \end{figure}
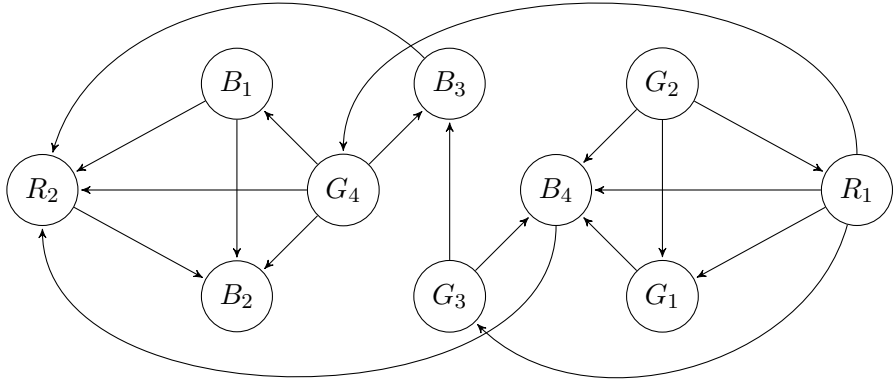
    
 {\bf Step~1.}       Let $H$ be the underlying graph with the  vertex set $$\{R_1, R_2, B_1, B_2, B_3, B_4, G_1, G_2, G_3, G_4\}$$
        in Figure~\ref{fig4}. 
Consider the following mapping $f$ from $S_m(n,k)$ to $H$.
        \begin{wst}
        	\item[\footnotesize $\bullet$] If the form of $\tau(\omega)$ is $e_0a_1$ then $f(\omega)=R_1$.
        	\item[\footnotesize $\bullet$] If the form of $\tau(\omega)$ is $e_1a_0$ then $f(\omega)=R_2$.
        	\item[\footnotesize $\bullet$] If $\omega\in B_1$ then $f(\omega)=B_1$.
        	\item[\footnotesize $\bullet$] If $\omega\in B_2$ then $f(\omega)=B_2$. 
        	\item[\footnotesize $\bullet$] If the form of $\tau(\omega)$ is $o_1o_0a_1$ then $f(\omega)=B_3$. 
        	\item[\footnotesize $\bullet$] If the form of $\tau(\omega)$ is $o_1e_0b_1$ then $f(\omega)=B_4$. 
        	\item[\footnotesize $\bullet$] If $\omega\in G_1$ then $f(\omega)=G_1$.
        	\item[\footnotesize $\bullet$] If $\omega\in G_2$ then $f(\omega)=G_2$. 
        	\item[\footnotesize $\bullet$] If the form of $\tau(\omega)$ is $o_0o_1a_0$ then $f(\omega)=G_3$.
        	\item[\footnotesize $\bullet$] If the form of $\tau(\omega)$ is $o_0e_1b_0$ then $f(\omega)=G_4$. 
        \end{wst}
Clearly, all sets of vertices of $S_m(n,k)$ mapped into the same vertex of $H$ are independent. For any vertex $\omega$ mapped into $G_1\cup G_2$, $\tau(\omega)$  has the form $o_0$; so, it may be adjacent only to a vertex $\omega'$ for which $\tau(\omega')$  has the form $o_1e_0$, or $e_0o_1$, or $o_0$. So, a neighbour of $\omega$ can only be mapped into  $G_1\cup G_2\cup R_1\cup B_4$.
Similarly,  a neighbour of a vertex from $B_1\cup B_2$ can only be mapped into $B_1\cup B_2\cup R_2\cup G_4$.
Also, if $\omega$ was mapped in $G_3$ then $\tau(\omega)$ has the form $o_0o_1a_0$ and for each of its neighbour $\omega'$, $\tau(\omega')$ has the form 
$e_0a_1$, or $o_1o_0a_1$, or $o_1a_0$. So, $\omega'$ must be in $R_1\cup B_3\cup B_4$. By similar arguments, any neighbour of a vertex from $B_3$ is mapped into $R_2\cup G_3\cup G_4$.

Finally, if  $\omega$ was mapped in $G_4$ then $\tau(\omega)$ has the form $o_0e_1b_0$ and for each of its neighbour $\omega'$,  $\tau(\omega')$ has the form $e_0a_1$, or $e_1a_0$, or $o_1$, or $o_1o_0a_1$. In either case, $\omega'$ was not mapped into $B_4$. So,  $f$ is a homomorphism from $S_m(n,k)$ to $H$. Then direct $H$ as in Figure~\ref{fig4}.

{\bf Step~2.} Note that
there are four shortcutting paths in $H$:  
	\begin{align*}
	 G_2\rightarrow R_1\rightarrow G_1\rightarrow B_4;&&
	 G_2\rightarrow R_1\rightarrow G_3\rightarrow B_4;\\
	 G_4\rightarrow B_1\rightarrow R_2\rightarrow B_2;&&
	 G_4\rightarrow B_3\rightarrow R_2\rightarrow B_2.
	\end{align*}

\noindent
Note that the shortcutting paths in the right column are shortcuts. Direct the edges of $S_m(n,k)$ as in $H$. Since the orientation of $H$ is acyclic, the orientation of $S_m(n,k)$ is also acyclic. We have to verify that the orientation contains no shortcuts.
	    
{\bf Step~3.}   Suppose that there exists a directed path $\omega_{G_2}\rightarrow\omega_{R_1}\rightarrow\omega_{G_1}\rightarrow\omega_{B_4}$ with $f(\omega_{X})=X$ for $X\in\{G_2,R_1,G_1,B_4\}$, and the path introduces a potential shortcutting edge, that is, $\omega_{G_2}\rightarrow\omega_{B_4}$.
Assume $\omega_{G_2}=x_1x_2\ldots x_n$ with $x_i\geq x_{i+1}+m$. Since the form of $\tau(\omega_{R_1})$ is $e_0x_1$ and $\omega_{G_2}\omega_{R_1}\in E(S_m(n,k))$, $\omega_{R_1}=x_2\ldots x_n y$ with $y\geq x_n+m$. Since $\omega_{R_1}\omega_{G_1}\in E(S_m(n,k))$, $\omega_{G_1}=x'_1x_2\ldots x_n$ with $x'_1\geq x_2 +m$ and $x'_1\neq x_1$. Then we have $\omega_{B_4}=zx'_1x_2\ldots x_{n-1}$ with $z\leq x'_1-m$. However,  $\omega_{G_2}\omega_{B_4}\in E(S_m(n,k))$ if and only if $x_1=x'_1$. Thus, $\omega_{G_2}\omega_{B_4}\notin E(S_m(n,k))$, a contradiction.

Then, we consider the shortcutting path $G_2\rightarrow R_1\rightarrow G_3\rightarrow B_4$. Suppose that there exists a directed path $\omega_{G_2}\rightarrow\omega_{R_1}\rightarrow\omega_{G_3}\rightarrow\omega_{B_4}$ with $f(\omega_{X})=X$ for $X\in\{G_2,R_1,G_3,B_4\}$ and also $\omega_{G_2}\rightarrow\omega_{B_4}$.
Assume $\omega_{G_2}=x_1x_2\ldots x_n$ with $x_i\geq x_{i+1}+m$. Then we have $\omega_{R_1}=x_2\ldots x_n y$ and $\omega_{G_3}= x_3\ldots x_n yz$ with $y\geq x_n+m$ and $z\leq y-m$. However, since $n\geq 6$, no neighbour of $\omega_{G_3}$ was mapped into $B_4$, a contradiction.

By similar arguments, no shortcut was mapped into the shortcutting paths $G_4\rightarrow B_1\rightarrow R_2\rightarrow B_2$ and 
$G_4\rightarrow B_3\rightarrow R_2\rightarrow B_2$. Hence, the orientation of $S_m(n,k)$ is shortcut-free and therefore it is semi-transitive. By Theorem~\ref{semi-trans-thm}, $S_m(n,k)$ is word-representable for all even $n\geq 6$ and $k\leq 2mn$.\end{proof}

	The remaining case of $n=4$ is proved in the following lemma.
	
	\begin{lem}\label{lemma-final}
		For $k\leq 5m$, $S_m(4,k)$ is word-representable.
	\end{lem}
	\begin{proof}
		By Lemma~\ref{obs}, it is sufficient to show that $S_m(4,5m)$ is word-representable.
		Let $A$, $B$ and $C$ be the following subsets of $A^4_m(5m,<)$:
		\begin{align*}
			A&:=\{x_1x_2x_3x_4\in A^4_m(5m,<):0\leq x_1\leq m-1 \text{ and } 4m\leq x_4\leq 5m-1\};\\
			B&:=\{x_1x_2x_3x_4\in A^4_m(5m,<):m\leq x_1\leq 2m-1 \text{ and } 4m\leq x_4\leq 5m-1\};\\
			C&:=\{x_1x_2x_3x_4\in A^4_m(5m,<):0\leq x_1\leq m-1 \text{ and } 3m\leq x_4\leq 4m-1\}.
		\end{align*}
		Since $x_4-x_1\geq 3m$ for $x_1x_2x_3x_4\in A^4_m(5m,<)$, $A\cup B\cup C$ is a partition of $A^4_m(5m,<)$. Note that any vertex in $A$ is an isolated vertex in $S^<_m(4,5m)$ and all edges in $S^<_m(4,5m)$ are between $B$ and $C$. Let $B_1=A\cup B$ and $B_2=C$. Then $S_m^<(4,5m)$ is bipartite with parts $B_1\cup B_2$.
		Let $G_1\cup G_2$ be a partition of $A^4_m(5m,>)$, where 
		$$G_i=\{x_4x_3x_2x_1\ |\ x_1x_2x_3x_4 \in B_i\} \text{ for } i=1,2.$$
		Then $S^{>}_m(4,5m)$ is also bipartite with parts $G_1\cup G_2$.

		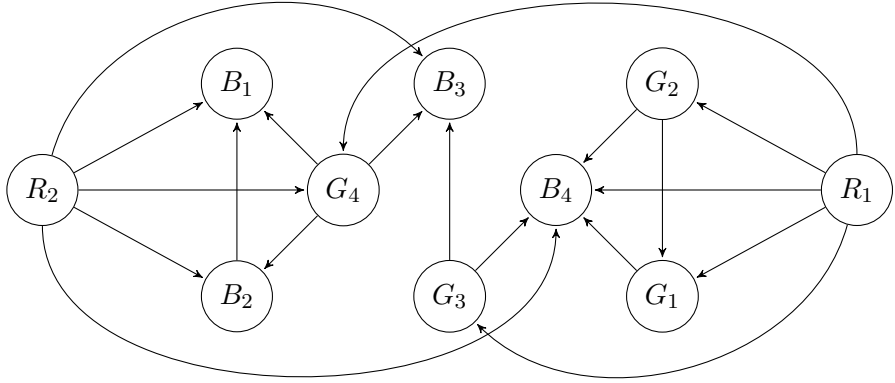
\begin{figure}
			\begin{center}        			
				\begin{tikzpicture}[->,>=stealth',shorten >=1pt,node distance=2cm,auto,main node/.style={circle,draw,align=center}]        			
					\node[main node] (3) {\small $B_3$};
					\node[main node] (2) [left of=3,xshift=-0.8284cm]{\small $B_1$};
					\node[main node] (4) [right of=3,xshift=0.8284cm]{\small$G_2$};
					\node[main node] (5) [below left of=3]{\small$G_4$};
					\node[main node] (6) [below right of=3]{\small$B_4$};
					\node[main node] (1) [right of=6,xshift=2cm] {\small $R_1$};
					\node[main node] (10) [left of=5,xshift=-2cm] {\small $R_2$};
					\node[main node] (7) [below left of=5]{\small$B_2$};
					\node[main node] (8) [below right of=5]{\small$G_3$};
					\node[main node] (9) [below right of=6]{\small$G_1$};
					
					\path
					(1) edge [bend right=90] node {} (5);
					\path
					(1) edge node {} (4);
					\path
					(1) edge node {} (6);
					\path
					(1) edge node {} (9);
					\path
					(1) edge [bend left=60] node {} (8);
					\path
					(10) edge [bend right=90] node {} (6);
					\path
					(10) edge node {} (7);
					\path
					(10) edge node {} (2);
					\path
					(10) edge node {} (5);
					\path
					(10) edge [bend left=60] node {} (3);
					\path
					(7) edge node {} (2);
					\path
					(4) edge node {} (9);
					\path
					(5) edge node {} (2);
					\path
					(5) edge node {} (3);
					\path
					(5) edge node {} (7);
					\path
					(4) edge node {} (6);
					\path
					(8) edge node {} (6);
					\path
					(9) edge node {} (6);
					\path
					(8) edge node {} (3);

				\end{tikzpicture}		
			\end{center}
			\caption{A graph $H$ and its orientation.}
			\label{fig6}
		\end{figure}
		
{\bf Steps~1,~2.}		In what follows, we use the same notation as in Lemma~\ref{leme2n} but a different orientation of $H$ given in Figure~\ref{fig6}. 
		Direct $S_m(4,5m)$ as in $H$ and note that 
		$H$ contains only two shortcutting paths of length at least 3:  
		\begin{align*}
			R_1\rightarrow G_2\rightarrow G_1\rightarrow B_4;&&
			R_2\rightarrow G_4\rightarrow B_2\rightarrow B_1.
		\end{align*}
		
{\bf Step~3.}		Suppose that there exists a directed path $\omega_{R_1}\rightarrow\omega_{G_2}\rightarrow\omega_{G_1}\rightarrow\omega_{B_4}$ with $f(\omega_{X})=X$ for all $X\in \{R_1,G_2,G_1,B_4\}$, and the path introduces a potential shortcutting edge, that is, $\omega_{R_1}\rightarrow\omega_{B_4}$.
		By definitions of $G_2$ and $G_1$, assume that $\omega_{G_2}=x_2x_3x_4x_5$ and $\omega_{G_1}=x_1x_2x_3x_4$, where $x_i\geq x_{i+1}+m$, $4m\leq x_1\leq 5m-1$ and $0\leq x_5\leq m-1$. Then, we have $\omega_{R_1}=x_3x_4x_5y$ with $y\geq x_5+m$ and $\omega_{B_4}=zx_1x_2x_3$ with $z\leq x_1-m$. However, $\omega_{R_1}$ is not adjacent to $\omega_{B_4}$ in $S_m(4,5m)$, a contradiction. So there are no shortcuts induced by $\{\omega_{R_1},\omega_{G_2},\omega_{G_1},\omega_{B_4}\}$. 

		Suppose  now that there exists a directed path $\omega_{R_2}\rightarrow\omega_{G_4}\rightarrow\omega_{B_2}\rightarrow\omega_{B_1}$ with $f(\omega_{X})=X$ for all $X\in \{R_2,G_4,B_2,B_1\}$ and also $\omega_{R_2}\rightarrow\omega_{B_1}$. Assume that 
$\omega_{B_1}=x_2x_3x_4x_5$ and $\omega_{B_2}=x_1x_2x_3x_4$, where $x_i\leq x_{i+1}+m$, $4m\leq x_5\leq 5m-1$ and $0\leq x_1\leq m-1$. 
 Then, we have $\omega_{R_2}=x_3x_4x_5y$ with $y\leq x_5-m$ and $\omega_{G_4}=zx_1x_2x_3$ with $z\geq x_1+m$. But then $\omega_{R_2}$ is not adjacent to $\omega_{G_4}$ in $S_m(4,5m)$, a contradiction. So, 
 there are no shortcuts induced by $\{\omega_{R_2},\omega_{G_4},\omega_{B_2},\omega_{B_1}\}$. Hence, the orientation of $S_m(4,5m)$ is shortcut-free and therefore it is semi-transitive.  By Theorem~\ref{semi-trans-thm}, $S_m(4,5m)$ is word-representable, as desired. \end{proof}

\section{Conclusion}
    In this paper, we introduce a novel approach to study word-representability of graphs with the help of homomorphisms. In the proof of Theorem~\ref{thmspn}, we find a homomorphism $f$ from $SP(n)$ to a 3-colorable graph $S_0(n-1,2)$. In the proof of Theorem~\ref{thmsnk}, we find a homomorphism $f$ from $S_m(n,k)$ to a graph with 5 vertices (for odd $n$) or a graph with 10 vertices (for even $n$).

    As the result, we have proved:

$\bullet$\ For $n=2$, $S_m(2,k)$ is word-representable if and only if $k\leq 3m$;

$\bullet$\ For $n=3$, $S_m(3,k)$ is word-representable if $k\leq 4m$;

$\bullet$\ For $n=4$, $S_m(4,k)$ is word-representable if $k\leq 5m$;

$\bullet$ For $n\ge 5,\ S_m(n,k)$ is word-representable if $k\leq 2mn$.

\noindent
We leave it as an open problem whether or not the ``if'' in the last three statements can be replaced by ``if and only if''. \\

\noindent
{\bf Acknowledgement.} The authors are grateful to the anonymous referees for their valuable comments.


\section*{Appendix: Proof of non-word-representability of $S_1(2,4)$}
Our proof uses the following lemmas.

\begin{lem}[\cite{KitaevSun}]\label{appendix1}
	Suppose that an undirected graph $G$ has a cycle $C=x_1x_2\cdots x_mx_1$, where $m\geq 4$ and the vertices in $\{x_1,x_2,\ldots,x_m\}$ do not induce a clique in $G$.  If $G$ is oriented semi-transitively, and $m-2$ edges of $C$ are oriented in the same direction then the remaining two edges of $C$ are oriented in the opposite direction.
\end{lem}

\begin{lem}[\cite{KitaevSun}]\label{appendix2}
	If $G$ is word-representable and $u$ is an arbitrary vertex in $G$, then there exists a semi-transitive orientation of $G$ with source $u$.
\end{lem} 

The proof is a tedious case-analysis with many similar procedures. So, we use some encoding for them that allows to shorten the text drastically.
By a ``line'' of a proof we mean a sequence of instructions that directs us in orienting a partially oriented graph and necessarily ends with detecting a shortcut. Each line of the proof is marked by its  number in {\bf bold} font. There are four types of instructions:

\begin{wst}
	\item[\footnotesize $\bullet$] ``B'' followed by ``$X\rightarrow Y$ (Copy $Z$)'' reads ``Branch on edge $XY$''. This means that we make a copy of the current graph (it is called Copy $Z$), direct the edge $X\rightarrow Y$ and proceed further. Note that Copy $Z$ will be considered later with an opposite orientation $Y\rightarrow X$.
	\item[\footnotesize $\bullet$] ``MC'' followed by a number $X$ means ``Move to Copy $X$''. Each line except for the first one must start with this instruction.  Moreover, this 
instruction is always followed by  ``$Y\rightarrow X$'' that orients a branching edge in the opposite way.
	\item[\footnotesize $\bullet$] One or two ``O'' followed by ``$X_i\rightarrow Y_i$'' together with ``(C'' followed by a cycle ``$Z_1-\cdots-Z_k$)''. This instruction means
``Orient the listed edges $X_i\rightarrow Y_i$ since otherwise the cycle $Z_1-\cdots-Z_k$ either becomes directed or contradicts Lemma~\ref{appendix1}.
	\item[\footnotesize $\bullet$] ``$S:X_1-\cdots-X_k$'' means ``Verify that the path $X_1-\cdots-X_k$'' induces a shortcut with the shortcutting edge $X_1\rightarrow X_k$. This instruction concludes each line of the proof.
\end{wst}
We refer to \cite{KitaevSun} for more details on the format of the proof below.
			
\begin{figure}
	\centering
	\includegraphics[width=80mm]{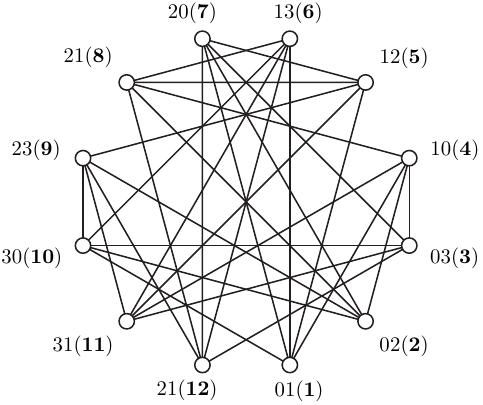}
	\caption{$S_1(2,4)$. The bold numbers are the labels of vertices in $S_1(2,4)$.}
	\label{fig5}
\end{figure}

\noindent
{\bf The proof:} We label the vertices in $S_1(2,4)$ in brackets in bold in Figure~\ref{fig5}. By Lemma~\ref{appendix2}, without loss of generality, assume vertex 6 is a source in a semi-transitive orientation. The rest of the proof goes as follows. \\[-3mm]

\begin{footnotesize}
	\noindent
	{\bf 1.} B8$\rightarrow$12 (Copy 2) B1$\rightarrow$10 (Copy 3) B7$\rightarrow$12 (Copy 4) O7$\rightarrow$ 1 (C1-7-12-6) O2$\rightarrow$10 O7$\rightarrow$2 (C1-10-2-7) 
O8$\rightarrow$2 (C2-8-12-7) S:6-8-2-10
	
	\noindent
	{\bf 2.} MC4 12$\rightarrow$7 O1$\rightarrow$7 (C1-7-12-6) O8$\rightarrow$2 O2$\rightarrow$7 (C2-8-12-7) O2$\rightarrow$10 (C1-10-2-7) S:6-8-2-10
	
	\noindent
	{\bf 3.} MC3 10$\rightarrow$1 B7$\rightarrow$12 (Copy 5) O7$\rightarrow$1 (C1-7-12-6) B9$\rightarrow$10 (Copy 6) O9$\rightarrow$5 O5$\rightarrow$1 (C1-10-9-5) O5$\rightarrow$8 (C1-6-8-5) O5$\rightarrow$11 (C1-6-11-5) O9$\rightarrow$11 (C5-11-9) O9$\rightarrow$2 O2$\rightarrow$8 (C2-9-5-8) O2$\rightarrow$7 (C2-8-12-7) O2$\rightarrow$4 O4$\rightarrow$1 (C1-7-2-4) 
O4$\rightarrow$11 (C1-5-11-4) S:9-2-4-11
	
	\noindent
	{\bf 4.} MC6 10$\rightarrow$9 
O12$\rightarrow$9 (C6-12-9-10) O2$\rightarrow$9 O7$\rightarrow$2 (C2-9-12-7) 
O5$\rightarrow$9 O7$\rightarrow$5 (C2-9-5-7) O5$\rightarrow$1 (C1-10-9-5) O5$\rightarrow$8 (C1-6-8-5) S:5-8-12-9
	
	\noindent
	{\bf 5.} MC5 12$\rightarrow$7 O1$\rightarrow$7 (C1-7-12-6) O10$\rightarrow$2 O2$\rightarrow$7 (C1-10-2-7) 
O8$\rightarrow$2 (C2-8-12-7) O8$\rightarrow$5 O5$\rightarrow$7 (C2-8-5-7) O1$\rightarrow$5 (C1-6-8-5) O10$\rightarrow$9 O9$\rightarrow$5 (C1-10-9-5) 
O9$\rightarrow$2 (C2-9-5-7) O9$\rightarrow$12 (C2-9-12-7) S:6-10-9-12
	
	\noindent
	{\bf 6.} MC2 12$\rightarrow$8 B1$\rightarrow$10 (Copy 7) B7$\rightarrow$12 (Copy 8) O7$\rightarrow$1 (C1-7-12-6) O2$\rightarrow$10 O7$\rightarrow$2 (C1-10-2-7) 
O2$\rightarrow$8 (C2-8-12-7) O5$\rightarrow$8 O7$\rightarrow$5 (C2-8-5-7) O5$\rightarrow$1 (C1-6-8-5) O9$\rightarrow$10 O5$\rightarrow$9 (C1-10-9-5) 
O2$\rightarrow$9 (C2-9-5-7) O12$\rightarrow$9 (C2-9-12-7) S:6-12-9-10
	
	\noindent
	{\bf 7.} MC8 12$\rightarrow$7 O1$\rightarrow$7 (C1-7-12-6) B9$\rightarrow$10 (Copy 9) 
O9$\rightarrow$12 (C6-12-9-10) O9$\rightarrow$2 O2$\rightarrow$7 (C2-9-12-7) 
O9$\rightarrow$5 O5$\rightarrow$7 (C2-9-5-7) O1$\rightarrow$5 (C1-10-9-5) O8$\rightarrow$5 (C1-6-8-5) S:9-12-8-5
	
	\noindent
	{\bf 8.} MC9 10$\rightarrow$9 O5$\rightarrow$9 O1$\rightarrow$5 (C1-10-9-5) O8$\rightarrow$5 (C1-6-8-5) O11$\rightarrow$5 (C1-6-11-5) O11$\rightarrow$9 (C5-11-9) O2$\rightarrow$9 O8$\rightarrow$2 (C2-9-5-8) O7$\rightarrow$2 (C2-8-12-7) O4$\rightarrow$2 O1$\rightarrow$4 (C1-7-2-4) 
O11$\rightarrow$4 (C1-5-11-4) S:11-4-2-9
	
	\noindent
	{\bf 9.} MC7 10$\rightarrow$1 B7$\rightarrow$12 (Copy 10) O7$\rightarrow$1 (C1-7-12-6) O2$\rightarrow$8 O7$\rightarrow$2 (C2-8-12-7) O10$\rightarrow$2 (C1-10-2-7) S:6-10-2-8
	
	\noindent
	{\bf 10.} MC10 12$\rightarrow$7 O1$\rightarrow$7 (C1-7-12-6) O10$\rightarrow$2 O2$\rightarrow$7 (C1-10-2-7) O10$\rightarrow$3 
 O2$\rightarrow$8 (C2-8-12-7) S:6-10-2-8
\end{footnotesize}


\end{document}